\documentclass[12pt, twoside, leqno]{article}

\usepackage{amsmath,amsthm}
\usepackage{amssymb}

\usepackage{enumitem}

\usepackage{graphicx}

\usepackage[T1]{fontenc}

\newtheorem{theorem}{Theorem}[section]
\newtheorem{corollary}[theorem]{Corollary}
\newtheorem{lemma}[theorem]{Lemma}
\newtheorem{proposition}[theorem]{Proposition}

\theoremstyle{definition}

\newtheorem{remark}[theorem]{Remark}

\numberwithin{equation}{section}

\begin{document}


\baselineskip=17pt


\title{Multipliers over Fourier algebras of ultraspherical hypergroups}

\author{ 
Reza Esmailvandi\\
Department of Mathematical Sciences\\
Isfahan University of Technology\\
	Isfahan 84156-83111, Iran\\
E-mail: r.esmailvandi@math.iut.ac.ir
\and
Mehdi Nemati\\
Department of Mathematical Sciences\\
Isfahan University of Technology\\
Isfahan 84156-83111, Iran\\
and\\
School of Mathematics\\
  Institute for Research in Fundamental
  Sciences (IPM)\\
  P. O. Box: 19395-5746, Tehran, Iran\\
E-mail: m.nemati@cc.iut.ac.ir
}

\date{}

\maketitle


\renewcommand{\thefootnote}{}



\renewcommand{\thefootnote}{\arabic{footnote}}
\setcounter{footnote}{0}


\begin{abstract}
Let $H$ be an ultraspherical hypergroup associated to a locally compact
group $ G $ and let $A(H)$  be the Fourier  algebra  of $H$.  For a left Banach $A(H)$-submodule $X$  of $VN(H)$, define $Q_X$ to be the norm closure of the linear span of the set $\{uf: u\in A(H), f\in X\}$  in $B_{A(H)}(A(H), X^*)^*$. We will show that  $B_{A(H)}(A(H), X^*)$ is a dual Banach space with predual $Q_X$, we characterize  $Q_X$ in terms of elements in  $A(H)$ and $ X$. Applications obtained on the multiplier algebra  $ M(A(H))$ of the Fourier algebra $ A(H)$. In particular,  we prove that $ G $ is amenable if and only if   $ M(A(H))= B_{\lambda}(H)$,
where $B_{\lambda}(H) $ is the reduced Fourier-Stieltjes algebra of $ H $. Finally, we investigate some characterizations for  an ultraspherical hypergroup  to be discrete.
\end{abstract}

\section{Introduction}
Let $G$ be a locally compact group and let $A(G)$ and $B(G)$ be the Fourier  and Fourier-Stieltjes algebras  of $G$ introduced by Eymard \cite{eym}. Let $M(A(G))$ denote the multiplier algebra of  $A(G)$.  Then we have the following inclusions
  $$
  A(G)\subseteq B(G)\subseteq M(A(G))
  $$
and  $\|v\|_{A(G)}=\|v\|_{B(G)}\geq \|v\|_M$ for all  $v \in A(G)$.
 It is known that if $G$ is amenable, then
 $B(G)= M(A(G))$ isometrically. Moreover,  it is known from Losert \cite{los1}  that $G$ is amenable,
 or equivalently $A(G)$ has a bounded approximate identity, whenever 
 the norms $\lVert\cdot\rVert_{B(G)}$ and  $\lVert\cdot\rVert_{M}$ are equivalent on $A(G)$.
As in the group case, the Fourier space $A(H)$ of a locally compact hypergroup $H$, plays an important role in the harmonic analysis. A class of hypergroups, called tensor hypergroups, whose Fourier space forms
 a Banach algebra under pointwise multiplication first
 appeared in \cite{am}. Another class, called ultraspherical hypergroups, was studied by
 Muruganandam \cite{murg2}. In this work, we study
 ultraspherical hypergroups through multipliers of $A(H)$, denoted $M(A(H))$.
 
Let ${\mathcal A}$ be a  Banach algebra, and let $X$ and
$Y$ be two right  Banach   ${\mathcal A}$-modules. 
Suppose that $B_{\mathcal A}(X, Y)$ is the Banach space of  bounded right ${\mathcal A}$-module
maps  with the operator norm denoted by $ \lVert \cdot\rVert_M $. In recent years, people have become interested in studying the properties of $B_{\mathcal A}(X, Y)$  for various classes of Banach algebras ${\mathcal A}$ and right  Banach   ${\mathcal A}$-modules $X$ and
$Y$; see for example \cite{hnr, hnr2, mi, miao}. 

In this paper, for a left Banach ${\mathcal A}$-submodule $X$ of ${\mathcal A}^*$ we study $B_{\mathcal A}({\mathcal A}, X^*)$ as a dual Banach space, paying  special attention to the Fourier algebra $A(H)$ of an ultraspherical hypergroup $H$ associated to a locally compact
group $ G $. 

In Section 2,  for a left Banach ${\mathcal A}$-submodule $X$ of ${\mathcal A}^*$, we show that  $B_{\mathcal A}({\mathcal A}, X^*)$ is a dual Banach space with predual $ Q_X $, where  $Q_X$ denote the norm closure of the linear span of the set  $\{af: a\in{\mathcal A}, f\in X\}$  in $B_{\mathcal A}({\mathcal A}, X^*)^*$. We will obtain   a characterization of $ Q_X$. 

In Section 3, we apply these results to Fourier algebra $A(H)$ of an ultraspherical hypergroup $H$. For the case of $X=C_\lambda^*(H)$, we show that the predual $Q_{C_\lambda^*(H)}$ of $M(A(H))$, the multiplier algebra of $A(H)$,
is equal to the closure of $L^1(H)$ in $M(A(H))$ under the multiplier norm. We also prove that $ G $ is amenable if and only if   $ M(A(H))= B_{\lambda}(H)$, where $B_{\lambda}(H) $ is the reduced Fourier-Stieltjes algebra of $H$.
In the case where $ A(H) $  is $w^*$-dense in  
$ M(A(H))$, we prove that $G$ is amenable if and only if  the norms $ \lVert \cdot\rVert_{A(H)} $and $ \lVert \cdot\rVert_{M} $ are equivalent on $ A(H)$.
 For the case of $X=C_\delta(H)$, we study the predual of   $B_{A(H)}({A(H)}, C_\delta(H)^*)$. These results generalize some results of \cite{miao} to ultraspherical hypergroups.

In Section 4, we shall define and study $ UCB(\widehat{H})$, called uniformly continuous functionals on $A(H)$. We will  focus in the relationship between $UCB(\widehat{H}) $ and other subspaces of $ VN(H) $.
We extend  various results of  \cite{lau1} to the context  of ultraspherical hypergroups.
For example, we prove that $ H $ is discrete if and only if $ UCB(\widehat{H}) =  C^{*}_{\lambda}(H)$. 
\section{The dual Banach space  $B_{\mathcal A}({\mathcal A}, X^*)$}

Let ${\mathcal A}$ be a  Banach algebra, and let $X$ and
$Y$ be right and left Banach  ${\mathcal A}$-modules, respectively. 
The {\it ${\mathcal A}$-module tensor product} of $X $ and $Y$ is the quotient
space $X\widehat{\otimes}_{\mathcal A} Y=(X\widehat{\otimes} Y )/ N$, where
$$
N=\langle x\cdot a\otimes y-x\otimes a\cdot y: x\in X, y\in Y, a\in {\mathcal A} \rangle,
$$
and $\langle \cdot\rangle$ denotes the closed linear span. It was shown in \cite{rifel} that
$$
B_{\mathcal A}(X, Y^*)\cong N^\perp\cong(X\widehat{\otimes}_{\mathcal A} Y)^*.
$$

Let $X$ be a left Banach ${\mathcal A}$-submodule of ${\mathcal A}^*$. In this section we show that $B_{\mathcal A}({\mathcal A}, X^*)$ is a dual Banach space and characterize its predual in terms of elements
in ${\mathcal A}$ and $X$. 
For every $a\in{\mathcal A}$ and $f\in X$, define the bounded linear functional $af$ on $B_{\mathcal A}({\mathcal A}, X^*)$ as follows:
$$
\langle af, T\rangle=\langle f, T(a)\rangle\quad (T\in B_{\mathcal A}({\mathcal A}, X^*)).
$$
Moreover, it is easy to see that $\|af\|_M\leq \|a\| \|f\|_X$. Now, we denote the  linear span of the set $\{af: a\in{\mathcal A}, f\in X\}$  by ${\mathcal A}X$ and define  $Q_X$ to be the norm closure of ${\mathcal A}X$ in $B_{\mathcal A}({\mathcal A}, X^*)^*$.

\begin{theorem}\label{2.2}
	Let ${\mathcal A}$ be a Banach algebra and let $X$ be a left Banach ${\mathcal A}$-submodule of ${\mathcal A}^*$. Then $B_{\mathcal A}({\mathcal A},  X^*)=(Q_X)^*$.
\end{theorem}
\begin{proof}
	Let
	$ J:\mathcal A \widehat{\otimes}X\rightarrow Q_{X} $ be defined
	by $J( \sum\limits_{i=1}^{\infty}a_i\otimes f_i)=\sum\limits_{i=1}^{\infty}a_i f_i$.
	Then it is clear that $ J $ is well defined and $\|J\|\leq 1$.
	As $ B(\mathcal A, X^{*})=(\mathcal A\widehat{\otimes} X)^{*}$, we have the adjoint operator
	$ J^*:(Q_X)^*\rightarrow B({\mathcal A}, X^{*}) $
	with 
	$ \lVert J^*\rVert\leq1$.
	Now, for each 
	$ T \in (Q_X)^* $,  we show that 
	$ J^*(T) \in  B_{\mathcal{A}}(\mathcal A, X^{*}).$
	Let $ a,b\in \mathcal A $ and $ f \in X $ . Then
	\begin{align*}
	\langle J^*(T)(ab), f\rangle
	&=\langle J^*(T),(ab)\otimes f\rangle
	=\langle T,(ab)f\rangle\\
	&=\langle T,a(bf)\rangle
	=\langle T,J(a\otimes(bf))\rangle\\
	&=\langle  J^*(T),a\otimes(bf)\rangle
	=\langle  J^*(T)(a),bf\rangle\\
	&=\langle  J^*(T)(a)\cdot b,f\rangle.
	\end{align*}
	Therefore,
	$ J^*(T)(ab)=  J^*(T)(a)\cdot b $
	for all $ a,b \in \mathcal{A}. $
	Thus,
	$ J^*(T) \in  B_{\mathcal{A}}(\mathcal A, X^{*}).$
	Let 
	$ T  \in  B_{\mathcal{A}}(\mathcal A, X^{*})$. Then the restriction   of $ T $
	to $ Q_X $ is in $ (Q_X)^* $ and we have
	\begin{eqnarray*}
	\langle J^*(T), \sum\limits_{i=1}^{\infty}a_i \otimes f_i  \rangle 
	&=&\langle T,  \sum\limits_{i=1}^{\infty}a_i f_i  \rangle 
	=\sum\limits_{i=1}^{\infty}\langle T,  a_i f_i  \rangle\\
	&=&\sum\limits_{i=1}^{\infty}\langle T(a_i) , f_i  \rangle
	=\langle T, \sum\limits_{i=1}^{\infty}a_i \otimes f_i  \rangle,
	\end{eqnarray*}
	for all $\sum\limits_{i=1}^{\infty}a_i \otimes f_i \in \mathcal A\widehat{\otimes} X$.
	It follows that $ J^*(T) = T $ and
	$ J^* $ is surjective.
	Since $ J(\mathcal A \widehat{\otimes}X) $
	is dense in $ Q_X, $
	by \cite[ Theorem 3.1.17]{meg}
	$ J^* $ is injective. Therefore, 
	$ J^* $ is a surjective isometry.
\end{proof}

\begin{theorem}\label{2.3}
	Let ${\mathcal A}$ be a Banach algebra and let $ X$ be a left Banach ${\mathcal A}$-submodule of ${\mathcal A}^*$. Suppose that  $f\in B_{\mathcal A}({\mathcal A}, X^*)$. Then $f\in Q_X$ if and only if there are sequences $(a_i)\subseteq {\mathcal A}$ and $(f_i)\subseteq X$ with $\sum_{i=1}^\infty \|a_i\|\|f_i\|<\infty$ such that $f=\sum_{i=1}^\infty a_if_i$ and 
	$$
	\|f\|_M=\inf\left\{\sum_{i=1}^\infty \|a_i\|\|f_i\|:  f=\sum_{i=1}^\infty a_if_i, \  \sum_{i=1}^\infty \|a_i\|\|f_i\|<\infty\right\}.
	$$
	
\end{theorem}
\begin{proof}
	By definition, each element of the form 
	$ \sum_{i=1}^{\infty} a_i f_i $, as  in the proof of Theorem \ref{2.2}, lies in 
	$ Q_X .$
	
	For the converse,  let $ I:{\mathcal A}\widehat{\otimes}_{\mathcal A} X\rightarrow Q_X $
	be defined by
	$$I(\sum_{i=1}^{\infty}a_i \otimes f_i+N) = \sum_{i=1}^{\infty}a_i  f_i.$$
	Then it is routine to check that $ I $ is well defined and $ \lVert I\rVert\leq1$. In fact, if $ \sum_{i=1}^{\infty}a_i \otimes f_i \in N $, then for each $ T \in B_{\mathcal A}({\mathcal A},  X^*)$, we have
	$$  \langle \sum_{i=1}^{\infty}a_i  f_i, T \rangle=\sum_{i=1}^{\infty}\langle T(a_i) ,  f_i \rangle= \langle\sum_{i=1}^{\infty} a_i \otimes f_i, T\rangle=0.$$
	Hence, $ I $ is well defined by duality.
	
	We know  from Theorem \ref{2.2} that $({\mathcal A}\widehat{\otimes}_{\mathcal A} X)^*= B_{\mathcal A}({\mathcal A},  X^*)=(Q_X)^* $. 
	It follows that
	$ I^*:(Q_X)^*\rightarrow ({\mathcal A}\widehat{\otimes}_{\mathcal A} X)^*$
	is bijective.
	Hence, $ I $ is surjective by
	\cite[ Theorem 3.1.22]{meg}.
	This proves first part of  the theorem.
	

	For the second part, let $ f \in Q_X $ and 
	$ \epsilon>0 $ be given. Then by first part of theorem, 
	there are sequences  $(a_i)\subseteq {\mathcal A}$ and $(f_i)\subseteq X$ such that $f=\sum_{i=1}^\infty a_if_i$ with $\sum_{i=1}^\infty \|a_i\|\|f_i\|<\infty$.
	Let $ \xi = \sum_{i=1}^\infty a_i \otimes f_i +N$. Then $  \langle T, f \rangle= \langle T, \xi \rangle $ for all $ T\in B_{\mathcal A}({\mathcal A},  X^*)$,  
	which implies that $ \lVert f\rVert_M=\lVert \xi\rVert. $
	Now, as a consequence of the definition of quotient norm, there exist sequences 
	$(b_i)\subseteq {\mathcal A}$ and $(h_i)\subseteq X$ such that $\sum_{i=1}^\infty \|b_i\|\|h_i\|<\lVert f\rVert_M +\epsilon$ and  $\xi=\sum_{i=1}^\infty b_i \otimes h_i +N.$
	Hence,
	$f=\sum_{i=1}^\infty b_i  h_i $
	on $ B_{\mathcal A}({\mathcal A},  X^*), $
	as required. This completes the proof.
\end{proof}

Suppose   that $X$ is a left Banach ${\mathcal A}$-module. Then $X^*$ is a right Banach ${\mathcal A}$-module with the following module action
$$
\langle m\cdot a ,f\rangle=\langle m, a\cdot f\rangle\quad(m\in X^*, f\in X, a\in {\mathcal A}).
$$
By the above notions it is not hard to see that, if $X$ is a left Banach ${\mathcal A}$-submodule of ${\mathcal A}^*$, then the map $$\iota: X^*\rightarrow B_{\mathcal A}({\mathcal A}, X^*),\quad m\mapsto m_L$$ is a
contractive linear map, where 
$m_L$ is given  by $ m_L(a)=m\cdot a$ for all $a\in{\mathcal A}$ and $\|m_L\|_M\leq \|m\|_X$. 
Thus, we can assume that $X^*\subseteq B_{\mathcal A}({\mathcal A},  X^*)$. Moreover, the adjoint map $\iota^*: B_{\mathcal A}({\mathcal A},  X^*)^*\rightarrow X^{**}$ is simply the restriction map, say $R$ and for every $a\in {\mathcal A}$, $f\in X$ and $m\in X^*$ we have
$$
\langle R(af), m\rangle=\langle af, m_L\rangle=\langle f, m_L(a)\rangle=\langle f, m\cdot a\rangle=\langle a\cdot f, m\rangle,
$$
which implies that $R(Q_X)\subseteq X$.

\begin{proposition}\label{p3}
	Let ${\mathcal A}$ be a Banach algebra and let $X$ be a left Banach ${\mathcal A}$-submodule of ${\mathcal A}^*$. Then  $R: Q_X\rightarrow X$ is surjective if and only if the norms $\|\cdot\|_X$ and $\|\cdot \|_M$ are equivalent on $X^*$.
	
\end{proposition}
\begin{proof}
	Let $ R $ be surjective.
	Then $ R^*: X^*\rightarrow (Q_X)^*$ is injective and $R^*(X^*)$
	is closed in $ (Q_X)^* $ by
	\cite[ Theorem 3.1.22]{meg}. Since $\lVert \cdot\rVert_M\leq \lVert \cdot\rVert_X$ on $X^*$, the Open Mapping theorem shows that the norms  $ \lVert \cdot\rVert_M $ and $ \lVert \cdot\rVert_X $ are equivalent on $ X^* $.
	
	Conversely, let the norms $ \lVert \cdot\rVert_M $ and $ \lVert \cdot\rVert_X $ are equivalent on $ X^*$. 
	Then $R^* $ is injective and $R^*(X^*) $ is closed in $ (Q_X)^* $.
	It follows from  \cite[ Theorem 3.1.17]{meg} and \cite[ Theorem 3.1.21]{meg} that 
	$R$ is surjective.
	
\end{proof}

For every $a\in {\mathcal A}$ we can regard $a$ as a functional on $X$. It follows that the  map $$\iota: {\mathcal A}\rightarrow B_{\mathcal A}({\mathcal A}, X^*),\quad a\mapsto a_L$$ is a
contractive linear map, where 
$a_L$ is given  by $ a_L(b)=ab$ for all $b\in{\mathcal A}$ and $\|a_L\|_M\leq\|a\|_X\leq \|a\|_{\mathcal A}$. This implies that ${\mathcal A}\subseteq B_{\mathcal A}({\mathcal A}, X^*)$.

Define $\widetilde{Q}_X$ to be the range of the linear map $\Gamma: {\mathcal A}\widehat{\otimes} X\rightarrow {\mathcal A}^*$ defined by $\Gamma (a\otimes f)=a\cdot f$. Then $\widetilde{Q}_X$ is a Banach space when equipped with the quotient norm from ${\mathcal A}\widehat{\otimes} X$. Moreover, $f\in \widetilde{Q}_X$ if and only if there are sequences $(a_i)\subseteq {\mathcal A}$ and $(f_i)\subseteq X$ with $\sum_{i=1}^\infty \|a_i\|\|f_i\|<\infty$ such that $f=\sum_{i=1}^\infty a_i\cdot f_i$.

\begin{theorem}\label{p4}
	Let ${\mathcal A}$ be a Banach algebra and let $X$ be a left Banach ${\mathcal A}$-submodule of ${\mathcal A}^*$. Then ${\mathcal A}$ is $w^*$-dense in $B_{\mathcal A}({\mathcal A}, X^*)$ if and only if $\widetilde{Q}_X$ is isometrically isomorphic to $Q_X$.
\end{theorem}
\begin{proof}
	Let ${\mathcal A}$ be $w^*$-dense in $B_{\mathcal A}({\mathcal A}, X^*).$ Then it follows from \cite[Proposition 2.6.6]{meg} that  the annihilator
	$(^{\perp}\mathcal{A})^{\perp}$ of $^{\perp}\mathcal{A}$ in  $ B_{\mathcal A}({\mathcal A}, X^*)$ can be identified with $ B_{\mathcal A}({\mathcal A}, X^*)=(Q_X)^*$, where 
	$$
	^{\perp}\mathcal{A}=\{f \in Q_X: \langle a_L, f \rangle =0 \ \text {for each } a\in {\mathcal A}\}.
	$$
	Hence, $\mathcal A$
	separates the points of 
	$ Q_X. $
	Now, define $ \Lambda:Q_X\rightarrow \widetilde{Q}_X $ by
	$$ \Lambda(\sum_{i=1}^{\infty} a_i f_i )=\sum_{i=1}^{\infty} a_i \cdot f_i.$$
	If  $ a \in \mathcal{A} $ is arbitrary, then for each sequences 
	$ (a_i)\subseteq \mathcal{A}$ and $(f_i) \subseteq X $ with $\sum_{i=1}^\infty \|a_i\|\|f_i\|<\infty,$  we have
	\begin{align*}
	\langle a_L,\sum_{i=1}^{\infty} a_i f_i \rangle
	=\sum_{i=1}^{\infty}\langle  aa_i, f_i \rangle
	=\sum_{i=1}^{\infty}\langle a,a_i \cdot  f_i \rangle
	=\langle a,\sum_{i=1}^{\infty} a_i \cdot  f_i \rangle.
	\end{align*}
	From this and the fact that $\mathcal A$
	separates the points of 
	$ Q_X, $
	we get that $ \Lambda $ is an isomorphism. Also, by Theorem \ref{2.3} it is an isometry. 
	
	Conversely, let $\widetilde{Q}_X$ be isometrically isomorphic to $Q_X.$
	Then $ \mathcal{A} $  separates the points of 
	$ Q_X$,
	which implies that 
	$  (^{\perp}\mathcal{A})^{\perp}= B_{\mathcal A}({\mathcal A}, X^*). $
	Again by \cite[Proposition 2.6.6]{meg}, ${\mathcal A}$ is $w^*$-dense in $ B_{\mathcal A}({\mathcal A}, X^*)$. 
\end{proof}

\section{The multiplier algebra $ M(A(H))$ and amenability}
A bounded linear operator on commutative Banach algebra ${\mathcal A}$ is called a multiplier if it satisfies $aT(b) = T(ab)$
for all $a, b\in{\mathcal A}$. We denote by ${\mathcal M}({\mathcal A})$ the space
of all multipliers for ${\mathcal A}$. Clearly ${\mathcal M}({\mathcal A})$ is a Banach algebra as a subalgebra of $B({\mathcal A})$ and ${\mathcal M}({\mathcal A})=B_{\mathcal A}({\mathcal A})$. For the general theory of multipliers
we refer to Larsen \cite{lar}. It is known that for a semisimple commutative Banach
algebra ${\mathcal A}$  every $T\in {\mathcal M}({\mathcal A})$ can be identified uniquely with a
bounded continuous function $\widehat{T}$ on $\Delta({\mathcal A})$, the maximal ideal space of ${\mathcal A}$.
Moreover, if  we denote
by $M({\mathcal A})$ the normed algebra of all bounded continuous functions $\varphi$ on
$\Delta({\mathcal A})$ such that $\varphi\widehat{\mathcal A}\subseteq \widehat{\mathcal A}$, then $M({\mathcal A})=\widehat{\mathcal M}({\mathcal A})$;  see \cite[Corollary 1.2.1]{lar}.

Let $H$ be an ultraspherical hypergroup associated to  a locally compact group
$G$
and a spherical
projector $\pi: C_c(G)\rightarrow C_c(G)$ which was introduced
and studied in \cite{murg2}. Let $A(H)$ denote the Fourier algebra corresponding to the hypergroup
$H$. A left Haar measure on $H$ is given by $\int_{H} f(\dot{x})d\dot{x}=\int_G f(p(x))dx$,  $f\in C_{c}(H)$, where  $p: G \rightarrow H$ is the quotient map. The Fourier space $A(H)$ is an algebra and is isometrically isomorphic to the subalgebra $A_{\pi }(G)=\lbrace u\in A(G)  :  \pi (u)=u\rbrace $ of $A(G)$ \cite[Theorem 3.10]{murg2}. Recall that the character space $\Delta(A(H))$ of $A(H)$ can be canonically identified with $H$. The Fourier algebra
$A(H)$ is semisimple, regular and Tauberian \cite[Theorem 3.13]{murg2}. As in the group case, let $\lambda$ also denote the left regular representation of $H$ on $L^2(H)$ given by 
$$
\lambda(\dot{x})(f)(\dot{y})=f(\check{\dot{x}}\ast\dot{y})\quad (\dot{x},\dot{y}\in H, f\in L^2(H))
$$
This can be extended to $L^1(H)$ by $\lambda(f)(g)=f*g$ for all $f\in L^1(H)$ and $g\in L^2(H)$. Let $C^*_{\lambda}(H)$ denote the completion of $\lambda(L^1(H))$
in $B(L^2(H))$ which is called the reduced $C^*$-algebra of $H$. The von Neumann algebra generated
by $\{\lambda(\dot{x}): \dot{x}\in H\}$ is called the von Neumann algebra of $H$, and is denoted by $VN(H)$. Note that $VN(H)$ is isometrically isomorphic to the dual of $A(H)$. Moreover, $A(H)$ can be considered as an ideal of $B_\lambda(H)$, where $B_\lambda(H)$ is the dual of $C_\lambda^*(H)$.
\begin{remark}\label{remb}
	As $A(H)$ is an ideal in $B_\lambda(H)$, there is a canonical $B_\lambda(H)$-bimodule structure on $VN(H)$. In particular, for
	$ f \in L^{1}(H) $ and $\phi\in B_{\lambda}(H) $, we obtain
	\begin{align*}
	\langle \phi \cdot \lambda(f), v \rangle
	= \langle  \lambda(f),  \phi v \rangle
	=\int f(\dot{x}) \phi(\dot{x}) v(\dot{x}) d\dot{x}
	=\langle  \lambda(\phi f),  v \rangle
	\end{align*}
	for all $ v \in A(H). $ This shows that $\phi \cdot \lambda(f)=\lambda(\phi f)\in \lambda(L^1(H))$. Since $\lambda(L^1(H))$ is norm dense in $C^*_\lambda(H)$, we conclude that $C^*_\lambda(H)$ is a $B_\lambda(H)$-bimodule.
\end{remark}
\begin{theorem}
	Let $H$ ba an ultraspherical hypergroup. Then
	$$
	M(A(H))=B_{A(H)}(A(H), C_\lambda^*(H)^*).
	$$
\end{theorem}
\begin{proof}
	Since $A(H)$ is commutative and semisimple, it suffices to show that ${\mathcal M}(A(H))=B_{A(H)}(A(H), B_\lambda(H))$. To prove this, first note that ${\mathcal M}(A(H))\subseteq B_{A(H)}(A(H), B_\lambda(H))$. Conversely, assume that $u\in A(H)$ has compact support. By regularity of $A(H)$, there exists $v\in A(H)$ such that $v(x)=1$ for $x\in {\rm supp}(u)$. Thus,  for each $T\in B_{A(H)}(A(H), B_\lambda(H))$, we have
	$$T(u)=T(vu)=vT(u).$$
	Since $A(H)$ is an ideal in $B_\lambda(H)$, we conclude that $T(u)\in A(H)$. Moreover, since the set of compactly supported
	elements in $A(H)$ is dense in $A(H)$, a simple approximation argument shows that $T(u)\in A(H)$ for all $u\in A(H)$. Therefore, $T\in {\mathcal M}(A(H))$ as required.
\end{proof}
Let $H$ ba an ultraspherical hypergroup and let $ f\in L^{1} (H)$. Define a linear functional on $ M(A(H)) $ by
$$ \langle f , \phi \rangle = \int f(\dot{x}) \phi (\dot{x}) d\dot{x} \quad (\phi \in M(A(H))).$$
Moreover, $ \lvert \langle f,\phi \rangle\rvert\leq\|f\|_{1}\|\phi\|_{\infty}\leq\|f\|_{1}\|\phi\|_{M}$ for all $ \phi \in M(A(H)).$ Therefore, $ f $ is in $M(A(H))^{*}$
and  
$$ 
\|f\|_{M}=\sup\left\{\left\lvert \langle f,\phi \rangle\right\rvert: \phi\in M(A(H)), \|\phi\|_M\leq 1\right\}\leq \|f\|_{1}. 
$$
Put
$$Q(H):=\overline{L^{1}(H)}^{\|.\|_{M}}\subseteq M(A(H))^{*}.$$

Next we  prove that $ M(A(H))$ is a
dual Banach space for any ultraspherical hypergroup $H$ .
\begin{theorem}\label{p5}
	Let $H$ ba an ultraspherical hypergroup. Then $ Q_{C^*_\lambda(H)}= Q(H)$ and so $ M(A(H))=Q(H)^{*}$.
\end{theorem}
\begin{proof}
	Suppose that $ f \in C_c(H)$. Using the regularity of $A(H)$, there exists $ u \in A(H) $
	such that $ u|_{\text{supp}(f)}\equiv 1$.
	Thus, $ f = uf $ is in $ Q_{C^{*}_{\lambda}(H)} $ 
	and 
	$ \langle uf, \phi\rangle =  \langle f, \phi\rangle =\int_{H} f(\dot{x}) \phi (\dot{x}) d\dot{x}$
	for all $ \phi \in M(A(H))$.
	Therefore, there is an isometry between the dense subspace of 
	$ Q_{C^{*}_{\lambda}(H)} $ 
	and a dense subspace of
	$ (L^{1}(H), \lVert \cdot\rVert _{M}) $.
	This shows that 
	$ Q_{C^{*}_{\lambda}(H)} $
	is the completion of 
	$L^{1}(H)$
	with respect to the norm $ \lVert \cdot\rVert_{M}$.
\end{proof}

\begin{theorem}\label{lema}
	Let $ H $ be an ultraspherical hypergroup on locally compact group $G$. Then $G$ is amenable if and only if $B_\lambda(H)=M(A(H))$.
\end{theorem}
\begin{proof}
	Suppose that $G$ is amenable. Then $B_\lambda(H)=M(A(H))$ by \cite[Theorem 4.2]{murg2}. Conversely, assume that $B_\lambda(H)=M(A(H))$. Then the constant function $1$ belongs to $B_\lambda(H)$. Since $A(H)$ is dense in $B_\lambda(H)$ with respect to the $\sigma(B_\lambda(H), C^{*}_{\lambda}(H))$-topology, there exists a net $(u_\alpha)$ in $A(H)$ such that
	$ u_{\alpha}\rightarrow 1$ in the $ \sigma(B_{\lambda}(H), C^{*}_{\lambda}(H)) $-topology and $c=\sup_\alpha\|u_\alpha\|_{A(H)}<\infty$.
	Choose $ f $ in $ C_{c}(H) $
	with $f\geq 0$ and
	$\|f\|_1 =1$. 
	For each $\alpha$, define
	$ u'_{\alpha}=f\ast u_\alpha$. 
	Notice first that $(u'_\alpha)\subseteq A(H)$ and
	$$
	\lVert u'_{\alpha}\rVert_{A(H)}
	\leq\lVert f\rVert_{1} \lVert u_{\alpha}\rVert_{A(H)}\leq c
	$$
	for all $ \alpha$. In fact, for each $g\in L^1(H)$ with $\|\lambda(g)\|_{C^*_\lambda(H)}\leq 1$, we have
	\begin{eqnarray*}
		|\langle f\ast u_\alpha, \lambda(g)\rangle|&=&|\int_H\int_Hf(\dot{y})u_\alpha(\check{\dot{y}}\ast \dot{x})g(\dot{x})d\dot{y}d\dot{x}|\\
		&=&|\int_Hf(\dot{y})\langle _{\dot{y}}u_\alpha, g\rangle d\dot{y}|\\
		&\leq& \int_H|f(\dot{y})|\|_{\dot{y}}u_\alpha\|_{A(H)} d\dot{y}\\
		&\leq& \|f\|_1\|u_\alpha\|_{A(H)}\leq c.
	\end{eqnarray*}
	Let $K\subseteq H$ be compact. Then the set $\{\lambda( _{\check{\dot{x}}}f): \dot{x}\in K\}$ form a compact subset of $C^*_\lambda(H)$, where the function $_{\check{\dot{x}}}f$ on $H$ is defined by $_{\check{\dot {x}}}f({\dot{y}})=f(\dot{x}\ast\dot{y})$ for all $\dot{y}\in H$.
	Since  $ u_{\alpha}\rightarrow 1$ in the $ \sigma(B_{\lambda}(H), C^{*}_{\lambda}(H))$-topology and the net $(u_\alpha)$ is bounded in $B_\lambda(H)$, the convergence is uniform on compact subsets of $C^{*}_{\lambda}(H)$. Hence,
	\begin{eqnarray*}
		u'_\alpha(\dot{x})=\langle {\check{u_\alpha}}, \lambda({_{\check{\dot{x}}}f})\rangle
		\rightarrow \langle 1, \lambda({_{\check{\dot{x}}}f})\rangle=
		\int_H{_{\check{\dot{x}}}}f(\dot{y})d\dot{y}=1
	\end{eqnarray*}
	uniformly on $K$, where $\check{u_\alpha}(\dot{x})=u_\alpha(\check{\dot{x}})$ for all $\dot{x}\in H$, and noticing that $\check{u_\alpha}\in B_\lambda(H)$ by \cite[Remark 2.9]{murg1}. 
	Again choose $ f $ in $ C_{c}(H) $
	with $f\geq 0$ and
	$\|f\|_1 =1$ and put $w_\alpha=f\ast u'_\alpha$ for all $\alpha$. Then $\|w_\alpha\|_{A(H)}\leq c$.
	Assume that 
	$ u\in A(H)\cap C_{c}(H). $ 
	Next, we show that 
	$ \lVert w_{\alpha}u-u \rVert_{A(H)}\rightarrow0$. In fact,
	if we put $ K=\text{supp}(f)^{\check{}}\ast \text{supp}(u)$,
	then for each
	$ \dot{x} \in \text{supp}(u) $
	we have
	\begin{eqnarray*}
	w_{\alpha}(\dot{x})&=&\int_{H}^{}f(\dot{y})u'_{\alpha}(\check{\dot{y}}\ast \dot{x})d\dot{y}\\
	&=&\int_{H}^{}f(\dot{y})({1}_{K}u'_{\alpha})(\check{\dot{y}}\ast \dot{x})d\dot{y}\\
	&=&(f\ast({1}_{K}u'_{\alpha}))(\dot{x}).
	\end{eqnarray*}
	Hence,
	$ uw_{\alpha}= u(f\ast({1}_{K}u'_{\alpha}))$, where ${1}_{K}$ denote the characteristic
	function of $K$.
	Similarly,
	$ u= u(f\ast{1}_{K})$.  Since
	$ \lVert{1}_{K}u'_{\alpha}-{1}_{K}\rVert_{2}\rightarrow0$, it follows that 
	$ \lVert uw_{\alpha}-u \rVert_{A(H)}\rightarrow0$. Finally, 
	since the net $ (w_{\alpha})$ is bounded and $ A(H)\cap C_{c}(H) $ is dense in
	$ A(H),$
	a straightforward approximation argument
	shows that
	$ \lVert uw_{\alpha}-u\rVert_{A(H)}\rightarrow0 $ for all $ u $ in $ A(H)$. Thus, $G$ is amenable by \cite[Theorem 4.4]{alagh}.
\end{proof}
\begin{corollary}
	Let $ H $ be an ultraspherical hypergroup on locally compact group $G$. Then the following hold.
	
	{\rm (i)} Let
	$ f \in M(A(H))^{*} $. 
	Then
	$ f \in Q(H)$
	if and only if there exist sequences 
	$ (u_{i})\subseteq A(H) $
	and
	$ (f_{i})\subseteq {C^{*}_{\lambda}(H)} $
	with
	$ \sum_{i=1}^{\infty}\lVert u_i\rVert_{A(H)} \lVert f_i\rVert_{C^{*}_{\lambda}(H)}<\infty  $ such that
	$f= \sum_{i=1}^{\infty} u_i f_i$ 
	and
	$$\lVert f\rVert_{M}= \inf \left\{\sum_{i=1}^{\infty} \lVert u_i\rVert_{A(H)}\lVert f_i\rVert_{C^{*}_{\lambda}(H)} :
	f= \sum\limits_{i=1}^{\infty} u_i f_i , ~ 
	\sum_{i=1}^{\infty}\lVert u_i\rVert_{A(H)} \lVert f_i\rVert_{C^{*}_{\lambda}(H)}<\infty\right\}.$$
	
	{\rm (ii)} $G$ is amenable if and only if  for any 
	$ f \in {C^{*}_{\lambda}(H)} $
	and 
	$ \epsilon>0 $
	there  exist sequences 
	$ (u_{i})\subseteq A(H) $
	and
	$ (f_{i})\subseteq {C^{*}_{\lambda}(H)} $
	such that
	$ f= \sum_{i=1}^{\infty} u_i f_i$ 
	on
	$ B_{\lambda}(H) $
	with
	$$\sum_{i=1}^{\infty}\lVert u_i\rVert_{A(H)} \lVert f_i\rVert_{C^{*}_{\lambda}(H)}<\lVert f\rVert_{C^{*}_{\lambda}(H)} + \epsilon.$$
\end{corollary}
\begin{proof}
	(i). It is an immediate consequence of Theorem \ref{2.3}.
	
	(ii). It follows from (i) that  the condition of (ii) is equivalent to   $ C^{*}_{\lambda}(H) = Q(H)$  (equivalently, $ B_{\lambda}(H)= M(A(H)) $). However  this is equivalent to  $ G $ being amenable by Lemma \ref{lema}.
	
\end{proof}

\begin{proposition}\label{2.5}
	Let $H$ be an ultraspherical hypergroup and let $X$ be a Banach $A(H)$-submodule of $VN(H)$ with $C^*_\lambda(H)\subseteq X$. Then $B_\lambda(H)$ is a subalgebra of 
	$B_{A(H)}(A(H), X^*)$ such that $\|\phi\|_M\leq\|\phi\|_{B_\lambda(H)}$ for all $\phi\in B_\lambda(H)$.
\end{proposition}
\begin{proof}
	Let $ u\in A(H) $ and $ \phi \in B_\lambda(H).  $
	Then 
	$ \phi u \in A(H) \subseteq VN(H)^*. $
	Thus $ \phi u \in X^*. $
	From this and the fact that
	$C^*_\lambda(H)\subseteq X,$ 
	we get that
	$$
	\lVert\phi u \rVert_{A(H)}
	=\lVert\phi u \rVert_{C^*_\lambda(H)}\leq \lVert\phi u \rVert_{X}\leq\lVert\phi  \rVert_{C^*_\lambda(H)}\lVert u \rVert_{A(H)}.$$
	Consequently,
	$ \lVert \phi\rVert_{M}\leq\lVert \phi\rVert_{B_\lambda(H)}. $
\end{proof}

Let $ H $ be an ultraspherical hypergroup. We say that $H$ has the approximation property if there is a net $ (u_\alpha) \subseteq A(H) $
such that
$ u_\alpha\stackrel{w^*}\rightarrow 1 $ in $M(A(H))$, i.e. in $ \sigma(M(A(H)), Q(H))$-topology.

For an ultraspherical hypergroup $H$, we put
\begin{center}
	$\overline{A(H)}^{w^*}:=$ the $w^*$-closure of $ A(H) $ in $ M(A(H)). $
\end{center}

\begin{proposition}\label{3.4}
	Let $ H $ be an ultraspherical hypergroup. Then $A( H) $ is $w^*$-dense in $M(A(H))$
	if and only if 
	$ H $ has the approximation property.
\end{proposition}
\begin{proof}
	We know that 
	$ 1\in M(A(H)). $
	Therefore, if 
	$A( H) $ is $w^*$-dense in $M(A(H)),$
	then $ 1\in  \overline{A(H)}^{w^*}. $
	Hence $ H $
	has the approximation property. For the converse, assume that $ H $ has the approximation property.
	Since $ L^{1}(H) $ 
	is dense in $Q(H)$,
	a simple approximation argument shows that
	$ \phi f \in Q(H)$
	for all 
	$ \phi \in M(A(H)) $ and 
	$ f \in Q(H)$.
	Now,  if there exists a net
	$ (u_\alpha) \subseteq A(H) $
	such that
	$ u_\alpha\xrightarrow{w^{*}} 1$ in $M(A(H))$,
	then for each
	$ \phi \in M(A(H))$,
	we have 
	$$\langle u_\alpha \phi, f \rangle 
	=\langle u_\alpha , \phi f \rangle
	\rightarrow \langle 1 , \phi f \rangle 
	=\langle \phi  f \rangle\quad (f  \in Q(H)). $$
	Consequently, $  \phi $ is $ w^{*} $-limit of   the net 
	$ (u_\alpha \phi)\subseteq A(H) $.
	Hence, 
	$\overline{A(H)}^{w^*}= M(A(H)),$
	as required.
\end{proof}
In what follows, for an ultraspherical hypergroup $H$, we put
\begin{center}
	$ Q^L(H):= $
	the Banach space of the restriction of elements in 
	$Q(H)$
	to
	$ \overline{A(H)}^{w^*}. $
\end{center}
\begin{proposition}\label{3.5}
	Let $ H $ be an ultraspherical hypergroup. Then the following hold.
	
	{\rm (i)}  $ \overline{A(H)}^{w^*}  $ is an ideal of $ M(A(H))$.
	
	{\rm (ii)} $ \overline{A(H)}^{w^*}=Q^L(H)^{*}. $
	
	Moreover, $ Q^L(H)$ is isometrically isomorphic to the completion of 
	$ L^{1}(G) $ with respect to the norm 
	$$\lVert f \rVert _{L}= \sup \left\{\left\lvert \int_{H} f(\dot{x}) \phi(\dot{x}) d\dot{x}\right\rvert : \phi \in \overline{A(H)}^{w^*},  \lVert\phi\rVert_{M}\leq1 \right\}.$$
	
\end{proposition}
\begin{proof}
	(i). If 
	$ \phi \in M(A(H)) $
	and 
	$ \psi \in \overline{A(H)}^{w^*}, $ 
	then there exists a net 
	$ (u_{\alpha})\subseteq A(H) $ 
	such that 
	$ u_{\alpha}\xrightarrow{w^{*}} \psi$. By the same argument as used in the proof of Proposition \ref{3.4} 
	and using the fact that
	$ (u_{\alpha}\phi) \subseteq A(H), $
	it is straightforward to conclude that 
	$ \phi \psi \in \overline{A(H)}^{w^*}. $
	Hence, $ \overline{A(H)}^{w^*} $
	is an ideal of 
	$ M(A(H))$.
	
	(ii). As an immediate consequence of the Hahn-Banach theorem the identity map 
	$ I:\overline{A(H)}^{w^*}\rightarrow Q^L(H)^{*}$ is an isometry. Let $ \psi \in  Q^L(H)^{*}  $
	with $ \lVert \psi\rVert=1$.
	Since 
	$ Q^L(H)$
	is a subspace of 
	$ (\overline{A(H)}^{w^*})^{*}$,
	we extend 
	$ \psi $
	to a linear functional $\phi$ on  
	$ (\overline{A(H)}^{w^*})^{*} $ with $\|\phi\|=1$.
	By the Goldstine's theorem,
	there is a net $(u_\alpha)$ in unit ball of 
	$\overline{A(H)}^{w^*}  $
	such that 
	$ u_{\alpha}\rightarrow\phi $
	in the 
	$ \sigma ((\overline{A(H)}^{w^*})^{**}, (\overline{A(H)}^{w^*})^{*})$-topology.
	In particular,
	$ \langle u_{\alpha}, f \rangle \rightarrow \langle \phi, f \rangle=\langle \psi, f \rangle $  for all 
	$ f \in Q(H). $
	Consequently, 
	$ \psi \in \overline{A(H)}^{w^*} $ and so 
	$ I $ is onto.
	Hence, 
	$ \overline{A(H)}^{w^*} $  
	is isometrically isomorphic to 
	$ Q^L(H)^{*}. $ Repeating the arguments in the proof of Theorem \ref{p5}, it is straightforward to prove the last statement.
\end{proof}

\begin{proposition}\label{3.7}
	Let $ H $ be an ultraspherical hypergroup on locally compact group $G$. Then the following hold.
	
	{\rm(i)} $ A(H) $  is $w^*$-dense in  
	$ M(A(H))$ if and only if the restriction map 
	$ R:Q(H)\rightarrow C^{*}_{\lambda}(H)$
	is injective.
	
	{\rm(ii)} The norms $ \lVert \cdot\rVert_{A(H)} $ and $ \lVert \cdot\rVert_{M} $ are equivalent on $ A(H) $ if and only if the restriction map 
	$ R:Q(H)\rightarrow C^{*}_{\lambda}(H)$ 
	is surjective.
	
	{\rm(iii)} If $ A(H) $  is $w^*$-dense in  
	$ M(A(H))$, then $G$ is amenable if and only if  the norms $ \lVert \cdot\rVert_{A(H)} $and $ \lVert \cdot\rVert_{M} $ are equivalent on $ A(H) $.
\end{proposition}
\begin{proof}
	(i).  Let $ A(H) $  be $w^*$-dense in  
	$ M(A(H)).$
	If 
	$f \in Q(H)$ with
	$ R(f)=0, $
	then  
	we have
	$ \langle f,u\rangle=\langle R(f), u\rangle=0$ for all  $ u \in A(H) $.
	Hence, a simple approximation argument gives that
	$ \langle R(f), \phi\rangle=0 $
	for all
	$ \phi \in M(A(H)). $
	Therefore, $ R $ is injective.
	Conversely, if 
	$ R $ is injective, then 
	$ B_{\lambda}(H)  $
	is $w^*$-dense in  
	$ M(A(H))$ by \cite[Theorem 3.1.17]{meg}.
	By an argument used  in the proof of  Proposition \ref{3.5}, we conclude that the identity map
	$ I: \overline{A(H)}^{w^*} \rightarrow   C^{*}_{\lambda}(H)^{*}   $
	is surjective. It follows that 
	$  \overline{A(H)}^{w^*}=  \overline{B_{\lambda}(H)}^{w^*}. $
	Hence, $ A(H) $ is $w^*$-dense in   $ M(A(H)).$
	
	(ii). Let $ \lVert \cdot\rVert_{A(H)} $ and $ \lVert \cdot\rVert_{M} $ be equivalent on $ A(H)$. We first show that the norm on  $ B_{\lambda}(H)  $  is equivalent to the multiplier norm.
	Let $ i:A(H)\rightarrow M(A(H)) $ be the inclusion map. Then $ i $ is bounded and has $ \lVert \cdot\rVert_{M} $-closed range. It follows from \cite[Theorem 3.1.21]{meg} that $ i^*(M(A(H))^*) $ is $w^*$-closed in $ A(H)^* $. Again, by \cite[Theorem 3.1.21]{meg}, $ i^{**}(A(H)^{**})$ is norm-closed in $ M(A(H))^{**} $. From this and the fact that 
	$ B_{\lambda}(H)  $ is norm-closed in $A(H)^{**}$, we conclude that the $ \lVert \cdot\rVert_{B_{\lambda}(H)} $-norm and the multiplier norm are equivalent on  $ B_{\lambda}(H)  $. Therefore,  $R $ is surjective  by Proposition \ref{p3}. 
	
	Conversely,  suppose that $ R  $ is surjective. Then it follows from Proposition \ref{p3} that the norms $ \lVert \cdot\rVert_{B_\lambda(H)} $ and $ \lVert \cdot\rVert_{M} $ are equivalent on $ B_\lambda(H)$ and hence on $A(H)$.

	(iii). Suppose first that $G$ is amenable. Then $A(H)$ has a bounded approximate identity by \cite[Theorem 4.4]{alagh}. It follows easily that  the norms 
	$ \lVert  \cdot  \rVert_{A(H)} $
	and
	$ \lVert \cdot \rVert_{M} $
	are equivalent on
	$A(H)$. Conversely, assume that  the norms 
	$ \lVert  \cdot  \rVert_{A(H)} $
	and
	$ \lVert \cdot \rVert_{M} $
	are equivalent on
	$ {A(H)} $.
	If  $ {A(H)} $ is $ w^{*} $-dense in  
	$ M(A(H))$, then by (i) and (ii)
	the restriction map is bijective. It follows that 
	$Q(H)$ is isometrically isomorphic to  
	$ C^{*}_{\lambda}(H)$, which implies that
	$ 1 \in M(A(H))=B_{\lambda}(H)$.
	Therefore,  $ G $ is amenable by Theorem \ref{lema}.
\end{proof}
\begin{remark}
	Identifying $\ell^1(H)$ with the subspace $\lambda(\ell^1(H))$ of $VN(H)$,
	we denote the norm closure of 
	$\ell^{1}(H)$
	in
	$ VN(H) $ by
	$C_\delta(H)$.
	Let $ f=\sum \alpha_{i}\lambda({\dot{x_{i}}})\in \ell^{1}(H)$
	and
	$ u \in A(H)$. Then
	$$u\cdot f=\sum \alpha_{i}  u(\dot{x_{i}})\lambda({\dot{x_{i}}})\in {C_\delta(H)},$$
	and
	$ \lVert u\cdot f\rVert_{{C_\delta(H)}} \leq \lVert u\rVert_{\infty} 
	\lVert f\rVert_{{C_\delta(H)}}
	\leq \lVert u\rVert_{A(H)} 
	\lVert f\rVert_{{C_\delta(H)}}$. Hence, 
	${C_\delta(H)}$
	is a Banach $ A(H)$-submodule of $ VN(H)$.
	Also, note that
	$ {C_\delta(H)}^{*}\subseteq \ell^{\infty}(H)$.
\end{remark}
\begin{proposition}\label{3.8}
	Let $ H $ be an ultraspherical hypergroup on locally compact group $G$. Then the following hold.
	
	{\rm(i)}
	$B_{A(H)}(A(H), {C_\delta(H)}^*)$
	consists of functions 
	$ \phi \in \ell^{\infty}(H) $
	such that the pointwise multiplication map
	$ T_{\phi}:A(H)\rightarrow  {C_\delta(H)}^*, u\mapsto \phi u$
	is a bounded operator.
	
	{\rm(ii)}
	$ Q_{C_\delta(H)}$
	is equal to the completion of 
	$\ell^{1}(H) $ with respect to the norm
	$$ \lVert f \rVert_M= \sup \left\{\left\lvert\sum f(\dot{x})\phi(\dot{x})\right\rvert:\phi \in B_{A(H)}(A(H), C_\delta(H)^*),  \lVert \phi\rVert\leq1\right\}.$$
	Furthermore,
	$M(A(H)) \subseteq B_{A(H)}(A(H), C_\delta(H)^*),$
	and the corresponding inclusion map
	is contractive.
\end{proposition}
\begin{proof}
	(i). Let
	$ \phi \in \ell^{\infty}(H) $ be such that
	$ T_{\phi}:A(H)\rightarrow C_\delta(H)^* $ 
	is a bounded linear operator. Then since 
	$$  T_{\phi}(u v)   = \phi u v= u  T_{\phi} (v) \quad (u, v \in A(H)),$$
	it follows that $ T_{\phi}\in B_{A(H)}(A(H) , C_\delta(H)^*). $
	For the reverse inclusion, let
	$ \phi \in B_{A(H)}(A(H), C_\delta(H)^*)$.
	Define
	$ \tilde{\phi}:H\rightarrow \mathbb{C}$ by $  \tilde{\phi}(\dot{x})=\langle\phi(u), \lambda(\dot{x})\rangle$,
	where 
	$ u $ denotes a function in 
	$ A(H)\cap C_{c}(H) $ with $ u(\dot{x})=1. $
	Then it is well defined. In fact, 
	if $ v $ is another function in 
	$ A(H)\cap C_{c}(H) $ such that
	$ v(\dot{x})=1$, then we put 
	$ K= \text{supp}(u)\cup \text{supp}(v) $ and choose
	$ w \in  A(H)\cap C_{c}(H) $ such that
	$ w|_{K}\equiv 1$. 
	Then
	\begin{eqnarray*}
	\langle\phi(u), \lambda(\dot{x})\rangle&=& \langle\phi(uw), \lambda(\dot{x})\rangle
= u(\dot{x})\langle\phi(w), \lambda(\dot{x})\rangle\\
&=& v(\dot{x})\langle\phi(w), \lambda(\dot{x})\rangle
	= \langle\phi(v w), \lambda(\dot{x})\rangle\\
	&=& \langle\phi(v ), \lambda(\dot{x})\rangle.
\end{eqnarray*}
	Observe next that if 
	$ u \in A(H) $,  $ \dot{x} \in H $ and 
	$v \in A(H)\cap C_{c}(H) $ with
	$ v(\dot{x})=1, $
	then
	\begin{eqnarray*}
	\langle\phi(u), \lambda(\dot{x})\rangle&=&v(\dot{x})\langle\phi(u), \lambda(\dot{x})\rangle=\langle\phi(uv), \lambda(\dot{x})\rangle\\
	&=&u(\dot{x})\langle \phi(v), \lambda(\dot{x})\rangle=u(\dot{x})\tilde{\phi}(\dot{x}).
	\end{eqnarray*}
	This shows that $ \phi= T_{\tilde{\phi}}$.
	
	(ii). Since 
	$C_\delta(H)$ is a Banach 
	$ A(H) $-submodule of 
	$ VN(H),$
	it follows from Theorem \ref{2.2} that 
	$$B_{A(H)}(A(H),
	C_\delta(H)^*)
	= Q_{C_\delta(H)}^{*}. $$
	Let 
	$ f \in \ell^{1}(H)$ be with finite support.
	Then 
	$ f = uf \in Q_{C_\delta(H)}$,
	where 
	$ u \in A(H) $ with $ u|_{\text{supp} (f)}\equiv 1. $
	Consequently,
	$$\langle \phi, f \rangle= \langle \phi,u f \rangle=\langle \phi(u), f \rangle
	=\sum \phi(\dot{x}) f(\dot{x}),  $$
	for all 
	$ \phi \in B_{A(H)}({A(H)},
	C_\delta(H)^*). $
	Hence, there is an isometry between the dense subspace of $\overline{\ell^1(H)}^{\lVert\cdot\rVert_{M}}$
	and a dense subspace of 
	$ Q_{C_\delta(H)} $.
	Therefore, $ Q_{C_\delta(H)}= \overline{\ell^1(H)}^{\lVert\cdot\rVert_{M}}.$
	
	Since 
	$ A(H)\subseteq  C_\delta(H)^{*}$ and $A(H)$ is an ideal in $M(A(H))$,
	it follows that
	$ \phi u \in C_\delta(H)^{*} $
	for all
	$ \phi \in M(A(H)) $
	and 
	$ u \in A(H). $
	This implies  that 
	$ M(A(H)) \subseteq B_{A(H)}(A(H), C_\delta(H)^*)$. Furthermore, 
	$$\lVert\phi u\rVert_{ C_\delta(H)}
	\leq \lVert\phi u\rVert_{A(H)}
	\leq \lVert\phi \rVert_{M}  \lVert u\rVert_{A(H)}.$$
	Hence, the  inclusion map  is contractive.
\end{proof}

\section{Introverted subspaces of $ VN(H) $ and discreteness}
Let $ H $ be an ultraspherical hypergroup associated to  a locally compact group
$G$. The    Arens product on  $ VN(H)^* $ is defined as following three steps. For $ u, v$ in $A(H) $, $ T $ in $ VN(H) $ and  $ m,n \in VN(H)^*, $ we define 
$ u\cdot T $,
$ m\cdot T \in VN(H) $
and
$ m \odot n \in VN(H)^{*} $
as follows:
$$ \langle u \cdot T , v\rangle=  \langle T , uv\rangle,\quad \langle m \cdot T , u \rangle = \langle m , u \cdot T \rangle,\quad \langle m\odot n , T\rangle = \langle m , n \cdot T \rangle.$$
A linear subspace $ X $   of $ VN(H) $ is called topologically invariant if $ u\cdot X \subseteq X $ for all $ u \in A(H) $.  
The topologically invariant subspace $ X $ of $ VN(H) $  is called topologically introverted  if  $ m \cdot T \in X$ for all $ m \in X^* $ and $ T \in X $. In this case, $X^*$
is a Banach algebra  with the multiplication induced by the 
Arens product $\odot$ inherited from $VN(H)^{*}$.
Let $W(\widehat{H})$
be the set of all 
$ T $ in 
$VN (H) $
such that the map 
$ u\mapsto u \cdot T $
of 
$ A(H) $
into
$VN (H) $
is weakly compact. Let $UCB(\widehat{H})  $ denote the closed linear span of 
$$ \{u \cdot T : u \in A(H) , T \in VN(H)\} .$$
The elements in $ UCB(\widehat{H}) $ are called uniformly continuous functionals on $ A(H) $. We also recall that, subspaces $W(\widehat{H})$ and $UCB(\widehat{H})$ of $VN (H)$ are both topologically introverted.

\begin{proposition}\label{p8}
	Let $ H $ be an ultraspherical hypergroup. Then 
	$ C^{*}_{\lambda}(H) \subseteq W(\widehat{H}) $.
\end{proposition}
\begin{proof}
	It suffices to prove that if
	$ f \in L^{1}(H) $,
	then 
	$ \lambda(f) \in W(\widehat{H}) $.
	Let 
	$ f \in L^{1}(H) $
	be fixed.
	Then by Remark \ref{remb}, for each
	$\phi\in B_{\lambda}(H) $, we have $ \phi \cdot \lambda(f)= \lambda(\phi f)$.
	Consider the map
	$ \phi  \mapsto \lambda (\phi  f) $
	from 
	$ B_{\lambda}(H) $
	into 
	$ VN(H) $.
	This map is continuous
	when
	$ B_{\lambda}(H) $
	has the
	$ \sigma(B_{\lambda}(H),C^{*}_{\lambda}(H))$-topology and
	$ VN(H) $
	has the weak topology.
	Indeed, let $ \Psi \in VN(H)^{*} $ and
	$ (\phi_{\alpha})\subseteq B_{\lambda}(H) $
	be a net  such that 
	$ \langle\phi_{\alpha}, T\rangle\rightarrow \langle\phi, T\rangle $
	for all 
	$ T\in C^{*}_{\lambda}(H) $.
	Then the restriction of 
	$  \Psi $ to
	$ C^{*}_{\lambda}(H) $
	is in $C^{*}_{\lambda}(H)^{*}=B_{\lambda}(H)$. Thus, there exists 
	$\psi \in B_{\lambda}(H) $
	such that
	$$ \langle\Psi , \lambda(h)\rangle =  \langle\psi , \lambda(h)\rangle
	=\int h(\dot{x})\psi (\dot{x}) d\dot{x} \quad  (h\in L^{1}(H)).
	$$
	Hence,
	\begin{align*}
	\langle \Psi, \lambda(\phi_{\alpha}f)\rangle
	& =\langle \psi, \lambda(\phi_{\alpha}f)\rangle
	=\int \phi_{\alpha}(\dot{x}) f(\dot{x})\psi(\dot{x}) d\dot{x}\\
	&=\langle\phi_{\alpha},\lambda(\psi f) \rangle
	\rightarrow
	\langle\phi, \lambda(\psi f)\rangle 
	\\
	& =\langle\psi, \lambda(\phi f)\rangle
	=\langle\Psi, \lambda(\phi f)\rangle.
	\end{align*}
	It follows that the set
	$ \{ \phi \cdot \lambda(f): \phi \in B_{\lambda}(H), \lVert\phi\rVert\leq1 \}$ is relatively compact in the weak topology of 
	$ VN(H)$.
	The rest of the proof follows from the fact that
	$ A(H)\subseteq B_{\lambda}(H)$.  
\end{proof}

\begin{proposition}\label{p10}
	Let $ H $ be an ultraspherical hypergroup. Then
	$C^{*}_{\lambda}(H)\subseteq UCB(\widehat{H}) $.
\end{proposition}
\begin{proof}
	Let $ f \in C_{c}(H)$.  By regularity of $ A(H)$,
	there exists $ u\in A(H) $ such that $ u|_{\text{supp}(f)}\equiv1 $. Therefore,
	\begin{eqnarray*}
	\langle u \cdot\lambda(f), v\rangle 
	=\langle\lambda(f),u v\rangle
	&=&\int f(\dot{x}) u(\dot{x}) v(\dot{x})d\dot{x}\\
	&=&\int f(\dot{x}) v(\dot{x})dt\\
	&=& \langle\lambda(f), v\rangle
	\end{eqnarray*}
	for all $ v \in A(H)$. This implies that 
	$ u \cdot \lambda(f)  =\lambda(f)$.
	Hence, 
	$ \lambda(f) \in UCB(\widehat{H}).$
	Consequently, $  C^*_\lambda(H)\subseteq UCB(\widehat{H}) $ by the density of $C_{c}(H) $
	in $  C^{*}_{\lambda}(H).$
\end{proof}

Let $ X $ be a closed topologically invariant subspace of
$VN(H)$
containing
$\lambda(\dot{e}).$
Then
$ m\in X^{*} $
is called a topologically  invariant mean 
on $X$
if:

(i) \ $\lVert  m \rVert =\langle m,     \lambda({\dot{e}}) \rangle  =1;$

(ii) \ $ \langle m, u \cdot T\rangle = u(\dot{e}) \langle m,T \rangle $ 
for all
$ T\in X  $
and
$ u \in A(H).$

We denote by $TIM(X)$ the set of all topologically invariant means on $X$.  
We also recall from Remark \ref{remb} that the space  $  C^*_\lambda(H)$ is an $ A(H)$-submodule of $ VN(H)$.
The following proposition is a consequence of 
\cite[Proposition 5.7]{dl} and
\cite[ Proposition 6.3]{lau1987}
and the fact that 
$ A(H)$  is a commutative $F$-algebra.

\begin{proposition}\label{pw}
	Let $ H $ be an ultraspherical hypergroup. Then the following hold.
	
	{\rm(i)}  The space
	$C^{*}_{\lambda}(H) $
	is a topologically introverted subspace of 
	$ VN(H) $.
	
	{\rm(ii)} $ W(\widehat{H}) $ admits a unique topologically invariant mean.
\end{proposition}

\begin{corollary}
	Let $ H $ be an ultraspherical hypergroup. Then $ H $ is discrete if and only if 
	$ \lambda(\dot{e})\in C^{*}_{\lambda}(H).$
\end{corollary}
\begin{proof}
	If $H$ is discrete, then $\ell^1(H)=L^1(H)$.
	Therefore,
	$\lambda(\dot{e}) \in C^{*}_{\lambda}(H).$
	Conversely, assume that
	$\lambda(\dot{e}) \in C^{*}_{\lambda}(H),$
	and 
	$m$
	denote the unique topologically invariant mean on 
	$W(\widehat{H})$. 
	Then $ \langle m, \lambda(\dot{e}) \rangle =1 $.
	It follows that 
	$ H $
	must be discrete by \cite[ Theorem 4.4(iv)]{kumar1}.
\end{proof}
\begin{lemma}\label{l1}
	Let $ H $ be an ultraspherical hypergroup and let $ R: VN(H)^*\rightarrow UCB(\widehat{H})^* $ be the restriction map. Then
	$ R: TIM(VN(H))\rightarrow TIM(UCB(\widehat{H})) $
	is a bijection.
\end{lemma}
\begin{proof}
	If 
	$ m_{1},m_{2} \in  TIM(VN(H)) $
	with
	$  m_{1}\neq m_{2}, $
	then
	there exists 
	$ T \in VN(H) $
	such that
	$ \langle m_{1},  T\rangle\neq \langle m_{2},  T\rangle. $
	Given $ u \in A(H)$ with $u(\dot{e})=1 $, we have
	$$ \langle m_{1},u\cdot  T\rangle = \langle m_{1},  T\rangle \neq \langle m_{2},  T\rangle = \langle m_{2},u\cdot  T\rangle. $$
	This implies that
	$ R( m_{1})\neq R( m_{2}), $
	and hence $ R $
	is injective.
	
	Suppose that
	$ \tilde{m}\in TIM(UCB(\widehat{H})).$
	Choose 
	$ u \in A(H) $
	with
	$ \lVert u\rVert_{A(H)} = u(\dot{e})=1$; see
	\cite[ Proposition 3.4]{kumar1}.
	Define 
	$ m $
	on
	$ VN(H)^{*} $
	by
	$$\langle m, T\rangle=\langle\tilde{m}, u\cdot T \rangle\quad (T\in VN(H)).$$
	Since
	$ \lVert u\rVert_{A(H)}=1,$
	it follows that
	$ \lVert m\rVert\leq 1. $
	Moreover,
	$$\langle m, \lambda(\dot{e})\rangle
	=\langle \tilde{m} , u \cdot \lambda(\dot{e})\rangle 
	=u(\dot{e})\langle \tilde{m} , \lambda(\dot{e})\rangle
	=\langle \tilde{m} , \lambda(\dot{e})\rangle =1. $$
	Therefore,
	$ \lVert m\rVert=1$. Furthermore, for each $v \in A(H)$ and $T\in VN(H),$ we have
	\begin{eqnarray*}
	\langle m, v \cdot T\rangle
	&=&\langle\tilde{m}, u\cdot(v \cdot T)\rangle
	=\langle\tilde{m},v\cdot(u \cdot T)\rangle\\
	&=&v(\dot{e})\langle\tilde{m}, u\cdot T\rangle
	=v(\dot{e})\langle m,  T\rangle.
	\end{eqnarray*}
	Consequently,
	$ m \in TIM(VN(H)). $
	Finally, if 
	$ T \in UCB(\widehat{H}),$
	then
	$$\langle R(m), T\rangle
	=\langle m, T\rangle
	=\langle\tilde{m}, u \cdot T\rangle
	=\langle\tilde{m}, T\rangle.$$
	Hence, $ R $ is surjective.
\end{proof}

\begin{proposition}
	Let $ H $ be an ultraspherical hypergroup. Then the following
	are equivalent.
	
	{\rm(i)} $ H $ is discrete.
	
	{\rm(ii)} $ UCB(\widehat{H}) =  C^{*}_{\lambda}(H)$.
	
	{\rm(iii)} There is a unique topologically invariant mean on $UCB(\widehat{H})$.
\end{proposition}

\begin{proof}
	{\rm(i)}$ \Rightarrow ${\rm(ii)}. Assume that $ H $ is discrete. Then for each $\dot{x}\in H$, the 
	characteristic function 
	$ {1}_{\dot{x}}$ is in 
	$ A(H) $; see \cite[ Proposition 2.22]{murg1}.
	Let 
	$ T \in VN(H) $ be fixed.
	Then for each 
	$ v \in A(H)$, we get
	\begin{align*}
	\langle{1}_{\dot{x}}\cdot T, v \rangle
	=\langle T, v {1}_{\dot{x}}\rangle
	=\langle T, v(\dot{x}) {1}_{\dot{x}}\rangle
	=v(\dot{x})\langle T,  {1}_{\dot{x}}\rangle.
	\end{align*}
	Hence, ${1}_{\dot{x}}\cdot T=\langle T,  {1}_{\dot{x}}\rangle \lambda(\dot{x})\in C^{*}_{\lambda}(H)$.        
	Let $u\in A(H)$. Since $A(H)\cap C_c(H)$ is dense in $A(H)$, we
	can suppose that $u$ has compact and hence finite support. Thus, $u$ 
	is a finite linear combination of characteristic functions on one point sets. Therefore, $u\cdot T\in C^{*}_{\lambda}(H)$. 
	It follows from 
	Proposition \ref{p10} that
	$ UCB(\widehat{H}) =  C^{*}_{\lambda}(H)$.
	
	{\rm(ii)}$ \Rightarrow ${\rm(iii)}.  
	If  $ UCB(\widehat{H})=C^{*}_{\lambda}(H)$, 
	then 
	$ UCB(\widehat{H}) \subseteq W(\widehat{H}) $
	by Proposition \ref{p8}.
	Let $m, n$ be topologically invariant means on $VN(H)$. Then $m=n$ when restricted to $W(\widehat{H})$ by Proposition \ref{pw}(ii). Since $UCB(\widehat{H}) \subseteq W(\widehat{H})$, we conclude that  $R(m)=R(n)$, and hence $m=n$ by Lemma \ref{l1}. 
	Again Lemma \ref{l1}, implies that there is a unique topological invariant mean on $UCB(\widehat{H})$.
	
	{\rm(iii)}$ \Rightarrow ${\rm(i)}. This follows from Lemma \ref{l1} and \cite[ Theorem $1.7$]{kumar2}.   
\end{proof}
It is shown in \cite[ Theorem $3.15$]{murg2} that $ B_{\lambda}(H) $ is a Banach algebra under pointwise multiplication.
As shown in Proposition \ref{pw}, $C^{*}_{\lambda}(H)  $ is topologically introverted. In particular,   $  C^{*}_{\lambda}(H)^*=B_{\lambda}(H) $ is a Banach algebra with the Arens Product. It is shown  in \cite[Proposition 5.3]{lau1} that the Arens product on $B_{\lambda}(G) $ is precisely the pointwise
product on it. Following we show that the same is also true for an ultraspherical hypergroup $H$.

\begin{proposition}
	Let $ H $ be an ultraspherical hypergroup. Then the Arens product and the pointwise multiplication on
	$ B_{\lambda}(H) $
	coincide.
\end{proposition}
\begin{proof}
	Let $\phi, \psi\in B_{\lambda}(H)$. Then for each $f\in L^1(H)$, we have $$\langle\phi\psi, \lambda(f)\rangle=\langle\phi, \lambda(\psi f)\rangle=\langle\psi, \lambda(\phi f)\rangle.$$ This shows that the  pointwise multiplication on $B_{\lambda}(H)$ is separately continuous in the $w^*$-topology. 
	Furthermore,  for each $ \psi \in  B_{\lambda}(H)  $,  the map $\phi\mapsto \phi \odot \psi  $ from $B_{\lambda}(H)  $  into $B_{\lambda}(H)$   is $ w^*$-$w^*$-continuous. Since $C^{*}_{\lambda}(H)\subseteq W(\widehat{H})$, it follows from \cite[Proposition 3.11]{dl}
	that the map $\phi\mapsto \psi \odot \phi $ is   continuous in the $w^*$-topology. 
	Therefore, the Arens product also is  separately continuous in the weak$^* $-topology.
	Since the Arens product and the pointwise multiplication on $ A(H) $ coincide  and $A(H)$ is
	$w^{*}$-dense in 
	$ B_{\lambda}(H)$, we conclude that 
	$\phi \odot \psi = \phi  \psi$ for all  $\phi, \psi \in  B_{\lambda}(H)$.
\end{proof}



\begin{thebibliography}{HD82}




\normalsize
\baselineskip=17pt

	
	\bibitem{alagh}
	{ M. Alaghmandan,}
	{Remarks on weak amenability of hypergroups},
	arXiv:1808.03805v1.
	
\bibitem{am}	M. Amini and A. Medghalchi, Fourier algebras on tensor hypergroups. In Banach algebras and their applications, volume 363 of Contemp. Math., pages 1-14. Amer. Math. Soc., Providence, RI, 2004.
	
	
	\bibitem{dl} H. G. Dales and A. T.-M. Lau, {The second duals of Beurling algebras,}  Mem. Amer. Math. Soc.   { 177} (2005), 1-191.
	
\bibitem{eym}
{P. Eymard,} {L'alg\'{e}bre de Fourier d'un groupe localement
	compact}, { Bull. Soc. Math. France,} {92} (1964),
181-236.
	
	
	
	\bibitem{hnr} Z. Hu, M. Neufang, Z.-J. Ruan, { Multipliers on a new class of Banach algebras, locally compact quantum groups,
		and topological centres},  Proc. Lond. Math. Soc.  {100}  (2010),  429-458. 
	
	\bibitem{hnr2} Z. Hu, M. Neufang, Z.-J. Ruan, {Completely bounded multipliers over locally compact
		quantum groups},  Proc. Lond. Math. Soc.   { 103} (2011), 1-39.
	
	
	\bibitem{mi} M. Ilie, Dual Banach algebras associated to the coset space of a locally compact group, J. Math. Anal. Appl. 467 (2018), 550-569.
	
	
	\bibitem{lar} { R. Larsen,}
	{An introduction to the theory of multipliers}, Springer-Verlag, New
	York-Heidelberg,  1971.
	
	\bibitem{lau1} { A. T.-M. Lau,}
	{Uniformly continuous functionals on the Fourier algebra of any locally compact group}, Trans.  Amer. Math. Soc.   { 251} (1979), 39-59.
	
	
	
	\bibitem{lau1987} { A. T.-M. Lau,}
	{Uniformly continuous functional on Banach algebras},  Colloq.  Math.  { LI} (1987), 195-205.
	
	\bibitem{los1}
	V. Losert, { Properties of the Fourier algebra that are equivalent to amenability}, { Proc. Amer. Math. Soc.}
	{92} (1984), 347-354.
	
	
	\bibitem{meg} { R.E.  Megginson,} {An  introduction to Banach space theory}, Springer-Verlag, New York, 1998.
	
	\bibitem{miao} T. Miao, {Predual of the multiplier algebra of $A_p(G)$ and amenability}, Canad. J. Math.  {56} (2004), 344-355. 
	
	\bibitem{murg1}
	{ V. Muruganandam,}
 {Fourier algebra of a hypergroup. $I$,}
	J. Austral. Math. Soc.  {82}
	(2007), 59-83.
	
	\bibitem{murg2}
	{ V. Muruganandam,}
	{Fourier algebra of a hypergroup. $II$.
		Spherical hypergroups},
	Math. Nachr.    {11}
	(2008), 1590-1603.
	
	
	
	\bibitem{rifel} M. A. Rieffel, {Multipliers and tensor products on $L_p$-spaces of locally compact groups},
	Studia Math.  {33} (1969), 71-82.
	
	
	\bibitem{kumar1}{ N. Shravan Kumar},
	{Invariant mean on a class of von Neumann algebras related to ultraspherical hypergroups,}
	Studia. Math.  {225} (2014), 235-247.
	
	\bibitem{kumar2} { N. Shravan Kumar},
	{Invariant mean on a class of von Neumann algebras related to ultraspherical hypergroups $II$,}
	Canad. Math. Bull.  {60} (2017), 402-410.


\end{thebibliography}
\end{document}